\documentclass{amsart}       

\usepackage[latin1]{inputenc}
\usepackage{amssymb,latexsym, geometry, amsfonts, multicol, color, graphicx, tikz, array, wrapfig}    
\usetikzlibrary{decorations.pathreplacing, trees}
\usetikzlibrary{decorations.pathmorphing}
\usepackage{mleftright}             
\usepackage{verbatim}       
\usepackage{amsmath}
\usepackage{geometry}                		                 		
\usepackage{parskip}   
\usepackage{graphicx}
\usepackage{tikz}
\usepackage{qtree}
\usepackage{caption}
\usepackage{subcaption}
\usepackage{float}
\usepackage{hyperref}
\usepackage{etex}
\usetikzlibrary{arrows,positioning,shapes.geometric,calc}

\newtheorem{theorem}{Theorem}[section]
\newtheorem{corollary}[theorem]{Corollary}
\newtheorem{lemma}[theorem]{Lemma}
\newtheorem{proposition}[theorem]{Proposition}

\newtheorem{notation}[theorem]{Notation}

\newtheorem{definition}[theorem]{Definition}
\newtheorem{remark}[theorem]{Remark}

\newtheorem{question}[theorem]{Question}

\newcommand{\revised}[1]{{\color{black} #1}}
\newcommand{\rerevised}[1]{{\color{black} #1}}
\begin{document}

\author{Francis Black}
 
 \author{Elizabeth Drellich *}
 \thanks{*Partially supported by NSF grant DMS--1248171 and NSF grant DMS--1143716}
 
 \author{Julianna Tymoczko**} \thanks {**Partially supported by NSF grant DMS--1248171}
  
\title{Valid plane trees: Combinatorial models for RNA secondary structures with Watson-Crick base pairs
}

\maketitle
\begin{abstract}
The combinatorics of RNA plays a central role in biology.  Mathematical biologists have several commonly-used models for RNA: words in a fixed alphabet (representing the primary sequence of nucleotides) and plane trees (representing the secondary structure, or folding of the RNA sequence).  This paper considers an augmented version of the standard model of plane trees, one that incorporates some observed constraints on how the folding can occur.  In particular we assume the alphabet consists of complementary pairs, for instance the Watson-Crick pairs A-U and C-G of RNA.  

Given a word in the alphabet, a valid plane tree is a tree for which, when the word is folded around the tree, each edge matches two complementary letters.  Consider the graph whose vertices are valid plane trees for a fixed word and whose edges are given by Condon, Heitsch, and Hoos's local moves.  We prove this graph is connected.

We give an explicit algorithm to construct a valid plane tree from a primary sequence, assuming that at least one valid plane tree exists.  The tree produced by our algorithm has other useful characterizations, including a uniqueness condition defined by local moves.  We also study enumerative properties of valid plane trees, analyzing how the number of valid plane trees depends on the choice of sequence length and alphabet size.  Finally we show that the proportion of words with at least one valid plane tree goes to zero as the word size increases.   We also give some open questions.
\end{abstract}



\pagestyle{myheadings}
\thispagestyle{plain}
\markboth{Black, Drellich, and Tymoczko}{Valid Plane Trees}

\section{Introduction}

DNA has as its core a fundamental combinatorial object: words in a finite sequence, on which permutations act, with biologically significant output.  Until fairly recently RNA was considered a mute cousin of DNA, important primarily as a stepping stone.  Unlike the famous double helix, it has a single strand.
Because this strand has many nucleotides, it tends to fold onto itself like a too-long piece of tape.  The base pairs created by this folding define its secondary structure.  And it turns out that these secondary structures (and the tertiary structures---namely the embedding of the secondary structure into 3-dimensional space) determine the functional aspects of RNA, for instance whether RNA transcribes and interprets the code in DNA, builds proteins, sends messages, or even even rewrites genetic code. 

In this paper we study a combinatorial model for RNA secondary structures consisting of  {\em plane trees}, which are rooted planar trees whose half-edges are ordered counterclockwise from the root, like those shown in Figure \ref{fig: plane trees}.    The plane tree model is well-established \cite{DiscreteTopologicalModels, Biophysics, CombinatoricsofRNAsecondarystructures} and plane trees are conveniently one of the many sets of objects enumerated by the Catalan numbers \cite{ParkingFunctionsandCn, NCPSurprisingLocations}.  Though other sets enumerated by Catalan numbers--including non-crossing matchings and balanced parentheses--are also used by mathematical biologists to model RNA secondary structures, plane trees have the advantage of depicting a simplified shape of the observed RNA secondary structure.  

\revised{However the standard plane tree model does not incorporate information about base pairs.  In this paper we assume an alphabet consisting of pairs of complementary letters and add the constraint that a letter can only pair with its complement.  For instance we might take the alphabet to consist of the four nucleotides \{A, C, G, U\} and the complementary pairs to be the Watson-Crick complements A-U and C-G.  Or we might take the alphabet to consist of codons, each of which is a sequence of three nucleotides corresponding to an amino acid, {and require that codons only bond with their perfect Watson-Crick complements.} }

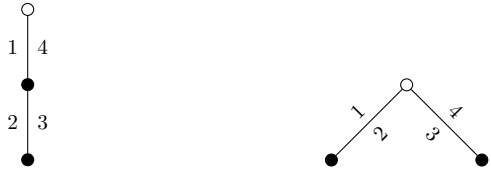
\begin{figure}[H]
\begin{center}
\begin{subfigure}[b]{0.3\textwidth}
\centering
  \begin{tikzpicture}[level distance=1cm,
level 1/.style={sibling distance=2cm},
level 2/.style={sibling distance=1cm}]
\tikzstyle{every node}=[circle, draw, scale=.8, inner sep=2pt]
\vspace{10pt}
\node (Root)[] {}
    child {
    node[fill] {}
    child { node[fill] {} 
	edge from parent 
	node[left, draw=none] {$2$}
	node[right, draw=none]  {$3$}    
    }
	edge from parent 
	node[left, draw=none] {$1$}
	node[right, draw=none]  {$4$}
};
\end{tikzpicture}   
\end{subfigure}
\begin{subfigure}[b]{0.3\textwidth}  
\centering
  \begin{tikzpicture}[sloped, level distance=1cm,
level 1/.style={sibling distance=2cm},
level 2/.style={sibling distance=1cm}]
\tikzstyle{every node}=[circle, draw, scale=.8, inner sep=2pt]
\vspace{10pt}
\node[inner sep=2pt] (Root)[] {}
    child {
    node[fill] {} 
    	edge from parent 
	node[ellipse, above, draw=none] {$1$}
	node[ellipse, below, draw=none]  {$2$}  
}
child {
    node[fill] {}
    	edge from parent 
	node[ellipse, below, draw=none] {$3$}
	node[ellipse, above, draw=none]  {$4$}  
};
\end{tikzpicture}
\end{subfigure}
\end{center}
\caption{The two plane trees with two edges}\label{fig: plane trees}
\end{figure}

A {\em primary structure} $P$ is a word in our alphabet.  Given a primary structure $P$, a plane tree $S$ is called {\em $P$-valid} if folding $P$ into the shape $S$ leaves every letter in $P$ matched with its complement (see Definition \ref{def: P-valid plane trees}).  Condon, Heitsch, and Hoos defined operations on the set of plane trees called local moves, which ``unzip" two adjacent edges and ``rezip" them in another arrangement (see Definition \ref{def: local moves}). We refine this to consider {\em valid} local moves, which require both initial and final plane tree to be valid plane trees.

\begin{figure}[H]
\begin{center}
\begin{subfigure}[b]{0.3\textwidth}
\centering
  \begin{tikzpicture}[level distance=1cm,
level 1/.style={sibling distance=2cm},
level 2/.style={sibling distance=1cm}]
\tikzstyle{every node}=[circle, draw, scale=.8, inner sep=2pt]
\vspace{10pt}
\node (Root)[] {}
    child {
    node[fill] {}
    child { node[fill] {} 
	edge from parent 
	node[left, draw=none] {$\overline{B}$}
	node[right, draw=none]  {$\overline{B}$}    
    }
	edge from parent 
	node[left, draw=none] {$B$}
	node[right, draw=none]  {$B$}
};
\end{tikzpicture}   
\end{subfigure}
\begin{subfigure}[b]{0.3\textwidth}  
\centering
  \begin{tikzpicture}[sloped, level distance=1cm,
level 1/.style={sibling distance=2cm},
level 2/.style={sibling distance=1cm}]
\tikzstyle{every node}=[circle, draw, scale=.8, inner sep=2pt]
\vspace{10pt}
\node[inner sep=2pt] (Root)[] {}
    child {
    node[fill] {} 
    	edge from parent 
	node[ellipse, above, draw=none] {$B$}
	node[ellipse, below, draw=none]  {$\overline{B}$}  
}
child {
    node[fill] {}
    	edge from parent 
	node[ellipse, below, draw=none] {$\overline{B}$}
	node[ellipse, above, draw=none]  {$B$}  
};
\end{tikzpicture}
\end{subfigure}
\end{center}
\caption{Only the plane tree on the right is valid for $P=B\overline{B}\overline{B}B$.}\label{fig: valid and non-valid plane trees}
\end{figure}
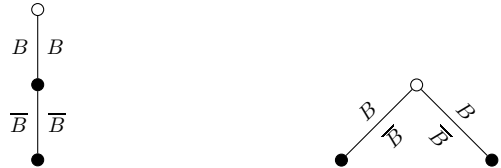

\revised{Our basic problem has non-biological antecedents, as well.  When the alphabet has exactly one complementary pair, the number of $P$-valid plane trees has a classical combinatorial interpretation as the number of ``polite" dinner party conversations between knights and ladies of the Round Table (where ``polite" encodes the notion of a noncrossing matching), as well as more serious mathematical applications to free probability theory and random matrix theory.  Kemp, Mahlburg, Rattan, and Smyth first studied this problem explicitly, giving partial results to enumerate and bound the set of P-valid plane trees for an alphabet with one complementary pair \cite{KMRS}.    Schumacher and Yan computed elegant formulas for the number of $P$-valid plane trees for specific families of sequences \cite{Schumacher2013}.}

We answer a collection of questions about primary sequences $P$ and their $P$-valid plane trees.  The first set assume we are given a primary structure and ask about the possible secondary structures:
\begin{itemize}
\item Can we determine if a valid plane tree for $P$ exists, and if so find it?   (Yes, by a greedy algorithm described in Definition \ref{def: greedy algorithm}.)
\item Given a primary sequence $P$ with at least one valid plane tree, how many orbits does it have under the action of valid local moves?  (Just one, by Theorem \ref{theorem: graph is connected}.)  
\item Valid local moves of type 2 are a kind of collapsing operation on valid plane trees, \rerevised{ defined in Definition \ref{def: local moves}. They allow us to create a directed graph whose vertices are the valid plane trees; for instance, there is a local move of type 2 from left to right in the plane trees in Figure \ref{fig: plane trees}.  On which valid plane trees can a valid local move of type 2 be performed?}  (On all trees except the tree produced by the greedy algorithm; see Corollaries \ref{theorem: non-greedy trees have type 2} and \ref{cor: unique sink}.)
\end{itemize}
The second set of questions asks about the possible primary sequences that satisfy various constraints: 
\begin{itemize}
\item How many sequences of length $2n$ have exactly $k$ valid plane trees?  (See results in Section \ref{sec: number of valid trees}, including Remark \ref{remark: bounding V(P)} and Proposition \ref{proposition: nothing between C_n and C_{n-1}}, which say none if $k>C_n$ or if $C_{n-1}<k<C_n$ respectively.  Also see recent work of Wagner \cite{Wagner}.)
\item How does the collection of valid plane trees depend on the choice of sequence length $2n$ and alphabet size $2m$? (See Section \ref{sec: number of valid trees} for a preliminary study.)
\item How many sequences $P$ have a valid plane tree?  (Relatively few---Theorem \ref{theorem: limiting zero} shows that the ratio of sequences with a valid plane tree to all sequences approaches zero as {sequence length} increases.)
\end{itemize}

\section{Valid plane trees and valid local moves}

A plane tree is a rooted tree  for which the children of each vertex are ordered.  When we draw plane trees, we depict the order by arranging the children left-to-right.  For example these two trees are different plane trees:
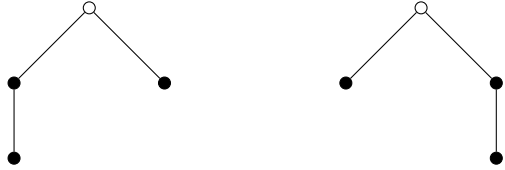
\begin{center}
  \begin{tikzpicture}[level distance=1cm,
level 1/.style={sibling distance=2cm},
level 2/.style={sibling distance=1cm}]
\tikzstyle{every node}=[circle, draw, scale=.8, inner sep=2pt]
\vspace{10pt}
\node (Root)[] {}
    child {
    node[fill] {} child {
    node[fill] {}
    	edge from parent 
}
    	edge from parent 
}
child {
    node[fill] {}
    	edge from parent 
};
\end{tikzpicture}   
\hspace{2cm}
  \begin{tikzpicture}[sloped, level distance=1cm,
level 1/.style={sibling distance=2cm},
level 2/.style={sibling distance=1cm}]
\tikzstyle{every node}=[circle, draw, scale=.8, inner sep=2pt]
\vspace{10pt}
\node[inner sep=2pt] (Root)[] {}
    child {
    node[fill] {} 
    	edge from parent 
}
child {
    node[fill] {} child {
    node[fill] {}
    	edge from parent
}
    	edge from parent 
};
\end{tikzpicture}
\end{center}
We label the half-edges of a plane tree {of size $n$} with the numbers $1$ through $2n$ starting on the left side of the leftmost edge adjacent to the root, and walking around the tree counter-clockwise.  The edges of the tree are called $e(i,j)$ where $i$ and $j$ are the labels on the left and right sides of the edge, respectively.  For example this tree contains the edges $e(1,2), e(3,6)$ and $e(4,5)$:
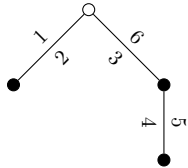
\begin{center}
  \begin{tikzpicture}[sloped, level distance=1cm,
level 1/.style={sibling distance=2cm},
level 2/.style={sibling distance=1cm}]
\tikzstyle{every node}=[circle, draw, scale=.8, inner sep=2pt]
\vspace{10pt}
\node[inner sep=2pt] (Root)[] {}
    child {
    node[fill] {} 
    	edge from parent 
	node[above, draw=none] {$1$}
	node[below, draw=none]  {$2$} 
}
child {
    node[fill] {} child {
    node[fill] {}
    	edge from parent
		node[above, draw=none] {$5$}
	node[below, draw=none]  {$4$}  
}
    	edge from parent 
		node[above, draw=none] {$6$}
	node[below, draw=none]  {$3$} 
};
\end{tikzpicture}
\end{center}

In the classical plane tree model of RNA-folding, the sequence of labels on the half-edges corresponds to the primary structure.  The plane tree model as presented by Condon, Heitsch, and Hoos  assumed that each half-edge corresponded to a string of six $G$s or six $C$s, all of which are paired  \cite[see Section 4]{Heitsch-Combinatorics-on-Plane-Trees}.  Hairpin loops, bulges, and other unmatched biological structures are omitted in the graph. (In fact RNA secondary structure always contains unpaired nucleotides since a hairpin loop with at least three unpaired nucleotides is needed for any pairs to bond {at all}.)



Plane trees are also important combinatorial objects.  We note three classical properties that are  useful in proofs.

\begin{proposition} \label{proposition: properties of plane trees}
Properties of plane trees: 
\begin{enumerate}
\item (Non-crossing condition) No edges $e(i,j)$ and $e(i',j')$ have $i<i'<j<j'$.
\item For each edge $e(i,j)$ the half-edges labeled $\{i+1,i+2,\ldots,j-1\}$ form a subtree, as do the half-edges labeled $\{1,2,\ldots,i-1,j+1,j+2,\ldots,n\}$.
\item For each edge $e(i,j)$ the half-edges $i$ and $j$ are of opposite parity.
\end{enumerate}
\end{proposition}

\begin{proof}
The non-crossing condition follows from the construction of plane trees  \cite[Section 4.4.2]{Bona-Handbook}.  The second property follows from the non-crossing condition and the third property follows from the second.
\end{proof}

Condon, Heitsch, and Hoos defined local moves to be the following operation on pairs of edges in a plane tree, which corresponds to an unfolding-and-refolding operation on nearby base pairs in a strand of RNA \cite[Definition 8]{Heitsch-Combinatorics-on-Plane-Trees}.

\begin{definition}
\label{def: local moves}
Let $i<j<i'<j'$ and fix a plane tree $S$. The two local moves are
\begin{itemize}
\item type 1: if $e(i,j)$ and $e(i',j')$ are adjacent edges in $S$ then  $e(i,j)$, $e(i',j')$ are replaced by $e(i,j')$, $e(j,i')$.
\item type 2: if $e(i,j')$ and $e(j,i')$ are adjacent edges in $S$ then  $e(i,j')$, $e(j,i')$ are replaced by $e(i,j)$, $e(i',j')$.
\end{itemize} 

A local move results in a new plane tree $S'$.
\end{definition}

Figure \ref{fig: local moves} shows the local moves.  Note that local moves can be performed on edges without any successively-labeled half-edges, namely with $i<j-1$ and $j<i'-1$ and $i'< j'-1$.  In other words there can be many other edges incident to the vertices in Figure \ref{fig: local moves}, including edges that come between the edges sketched in the schematic.

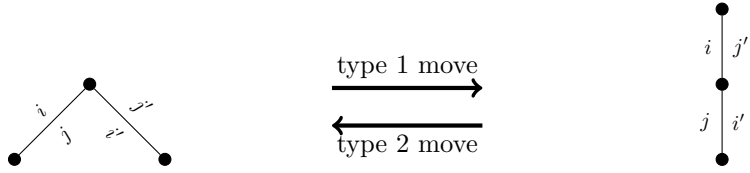
\begin{figure}[H]
\begin{center}
\begin{subfigure}[b]{0.3\textwidth}  
\centering
  \begin{tikzpicture}[sloped, level distance=1cm,
level 1/.style={sibling distance=2cm},
level 2/.style={sibling distance=1cm}]
\tikzstyle{every node}=[circle, draw, scale=.8, inner sep=2pt]
\vspace{10pt}
\node[inner sep=2pt] (Root)[fill] {}
    child {
    node[fill] {} 
    	edge from parent 
	node[ellipse, above, draw=none] {$i$}
	node[ellipse, below, draw=none]  {$j$}  
}
child {
    node[fill] {}
    	edge from parent 
	node[ellipse, below, draw=none] {$i'$}
	node[ellipse, above, draw=none]  {$j'$}  
};
\end{tikzpicture}
\end{subfigure}
\begin{subfigure}[b]{0.2\textwidth}
\centering
\begin{tikzpicture}
\draw [->, ultra thick] (-1,0.5)--(1,0.5);
\draw [<-, ultra thick] (-1,0)--(1,0);
\node [above]  at (0,.5) {type 1 move};
\node [below]  at (0,0) {type 2 move};
\end{tikzpicture}
\end{subfigure}
\begin{subfigure}[b]{0.3\textwidth}
\centering
  \begin{tikzpicture}[level distance=1cm,
level 1/.style={sibling distance=2cm},
level 2/.style={sibling distance=1cm}]
\tikzstyle{every node}=[circle, draw, scale=.8, inner sep=2pt]
\vspace{10pt}
\node (Root)[fill] {}
    child {
    node[fill] {}
    child { node[fill] {} 
	edge from parent 
	node[left, draw=none] {$j$}
	node[right, draw=none]  {$i'$}    
    }
	edge from parent 
	node[left, draw=none] {$i$}
	node[right, draw=none]  {$j'$}
};
\end{tikzpicture}   
\end{subfigure}

\end{center}
\caption{Local moves}
\label{fig: local moves}
\end{figure}

We define plane trees for sequences other than the sequence that Condon, Heitsch, and Hoos used.   We are motivated by the  example of RNA but an arbitrary alphabet consisting of complementary pairs is more mathematically natural.

\revised{\begin{definition}
A set $\mathcal{A}$ is a complementary alphabet if for every letter $B\in \mathcal{A}$ there is a unique complement $\overline{B}\in \mathcal{A}$ that is distinct from $B$.  If $\mathcal{A}$ is a complementary alphabet then taking the complement is an involution on $\mathcal{A}$ with $\overline{\overline{B}}=B$ for each $B \in \mathcal{A}$.
\end{definition}}

Though the collection of nucleotides $\{A,C,G,U\}$ is our foundational example, it is in fact not a complementary alphabet.  The Watson-Crick pairs $C-G$ and $A-U$ form the strongest and second-strongest bonds respectively.  However so-called ``wobble pairs" $G-U$ also bond in nature, though more weakly than the others. 

The next definition introduces one of the key objects in this paper: valid plane trees. Valid plane trees improve upon the plane tree model by considering whether, using only perfect complements, a primary structure could actually form the secondary structure of a particular plane tree.  

\revised{\begin{definition}[$P$-valid plane trees]
\label{def: P-valid plane trees}
Let $S$ be a plane tree with $n$ edges and let $P=p_1p_2\cdots p_{2n}$ be a word in a complementary alphabet $\mathcal{A}$.  We say  $S$ is a  $P$-valid plane tree if for every edge $e(i,j)$ in $S$ the letters $p_i$ and $p_j$ are a complementary pair in $\mathcal{A}$.
We may refer to $S$ as a valid plane tree if $P$ is clear from context.  We denote the set of all $P$-valid plane trees $\mathcal{V}(P)$.
\end{definition}}

We only consider local moves that send  $P$-valid plane trees to other $P$-valid plane trees. More precisely we have the following definition.

\begin{definition}
\label{def: valid local moves}
Fix a primary sequence $P$ and let $S$ be a $P$-valid plane tree, i.e. $S\in \mathcal{V}(P)$. A valid local move on $S$ is a local move on $S$ such that the resulting plane tree $S'$ is also in $\mathcal{V}(P)$.
\end{definition}

The following two claims are immediate  since local moves are invertible.

\begin{proposition} A local move of type 1 is valid if and only if the inverse local move of type 2 is valid.
\end{proposition}

\begin{proposition} The following conditions are necessary and sufficient to guarantee the existence of a valid local move with respect to $P=p_1p_2\cdots p_{2n}$

\begin{itemize}
\item There is a valid type 1 local move on adjacent edges $e(i,j)$ and $e(i',j')$ if and only if for some $B, \overline{B}\in \mathcal{A}$ the letters $p_i = p_{i'} = B$ and $p_j=p_{j'}=\overline{B}$. 
 \item There is a valid type 2 move on adjacent edges $e(i,j')$ and $e(j,i')$ if and only if  for some $B, \overline{B}\in \mathcal{A}$ the letters  $p_i=p_{i'}=B$ and $p_j=p_{j'}=\overline{B}$. \end{itemize}
\end{proposition}

In fact the property of being $P$-valid is inherited by subtrees in the following sense.  

\begin{proposition}\label{proposition: valid subtrees}
Suppose $T$ is a $P$-valid tree and that $T$ contains an edge $e(i,j)$ and denote by $v$ the vertex of $e(i,j)$ farthest from the root. 
\begin{itemize} \item The subgraph $T'$ induced from $e(i,j)$ together with all descendants of $v$ is equivalently characterized as the subgraph containing half-edges $i,i+1,\ldots,j-1,j$.  \end{itemize}

Define $T''$ to be the subgraph of $T$ obtained by erasing $T'$ from $T$.  Define $P'$ to be the subsequence of $P$ consisting of the labels on the half-edges $i,i+1,i+2,\ldots,j-1,j$ and $P''$ to be the subsequence of $P$ consisting of the labels on the half-edges $1,2,\ldots,i-1,j+1,j+2,\ldots,n$ retaining the original indexing for notational convenience. 
\begin{itemize} \item Then $T'$ is a $P'$-valid tree and $T''$ is a $P''$-valid tree. \end{itemize}\end{proposition}

\begin{proof}
By our convention of labeling half-edges in a plane tree, the half-edges descended from $v$ in $T$ are labeled $i+1,i+2,\ldots,j-1$.  (Figure \ref{figure: edge partitioning tree} shows a schematic.)  The descendants of a vertex in a rooted tree form a subtree, so $T'$ is a $P'$-valid tree.  The graph $T''$ is connected by construction and is formed of the half-edges labeled $1,2,\ldots,i-1,j+1,j+2,\ldots,n$ so $T''$ is a $P''$-valid tree.
\end{proof}

The condition of being $P$-valid {respects the fact that any two plane trees of size $n$}  either have some overlap or are very close to having some overlap, as the following proposition makes precise. 

\begin{proposition}\label{proposition: sharing an edge}
Given two trees $T, T' \in \mathcal{V}(P)$ either $T$ and $T'$ share an edge or there exists a valid local move on one, say from $T'$ to $T''$, such that $T$ and $T''$ share an edge.  
\end{proposition}

\begin{proof}
Suppose $T$ and $T'$ have no common edges.  Suppose $e(i,i+1)$ is a leaf of $T$.  Since it is not a leaf in $T'$ the half-edges labeled $i$ and $i+1$ have one of the following configurations in $T'$: 
\begin{figure}[H]
\begin{center}
\begin{subfigure}[b]{0.3\textwidth}  
\centering
\begin{tikzpicture}[level distance=1cm,
level 1/.style={sibling distance=2cm},
level 2/.style={sibling distance=1cm}]
\tikzstyle{every node}=[circle, draw, scale=.8, inner sep=2pt]
\vspace{10pt}
\node (Root)[fill] {}
    child {
    node[fill] {}
    child { node[fill] {} 
	edge from parent 
	node[left, draw=none] {$i + 1$}
	node[right,draw=none]  {$k$}    
    }
	edge from parent 
	node[left, draw=none] {$i$}
	node[right, draw=none]  {$j$}
};
\end{tikzpicture}  
\caption{$i$ and $i+1$ on left \label{fig:samerowleft}}   
\end{subfigure}
\begin{subfigure}[b]{0.3\textwidth}
\centering
  \begin{tikzpicture}[level distance=1cm,
level 1/.style={sibling distance=2cm},
level 2/.style={sibling distance=1cm}]
\tikzstyle{every node}=[circle, draw, scale=.8, inner sep=2pt]
\vspace{10pt}
\node (Root)[fill] {}
    child {
    node[fill] {}
    child { node[fill] {} 
	edge from parent 
	node[left, draw=none] {$j$}
	node[right, draw=none]  {$i$}    
    }
	edge from parent 
	node[left, draw=none] {$k$}
	node[right, draw=none]  {$i+1$}
};
\end{tikzpicture}
\caption{$i$ and $i+1$ on right \label{fig:samerowright}}   
\end{subfigure}
\begin{subfigure}[b]{0.3\textwidth}  
\centering
  \begin{tikzpicture}[sloped, level distance=1cm,
level 1/.style={sibling distance=2cm},
level 2/.style={sibling distance=1cm}]
\tikzstyle{every node}=[circle, draw, scale=.8, inner sep=2pt]
\vspace{10pt}
\node[inner sep=2pt] (Root)[fill] {}
    child {
    node[fill] {} 
    	edge from parent 
	node[ellipse, above, draw=none] {$j$}
	node[ellipse, below, draw=none]  {$i$}  
}
child {
    node[fill] {}
    	edge from parent 
	node[ellipse, below, draw=none] {$i + 1$}
	node[ellipse, above, draw=none]  {$k$}  
};
\end{tikzpicture}
\caption{$i$ and $i+1$ on peak \label{fig:diffrowV}} 
\end{subfigure}
\end{center}
\caption{}
\end{figure}
Since $e(i,i+1)$ is an edge of a valid plane tree, the half-edges $i$ and $i+1$ are labeled $B$ and $\overline{B}$ respectively and so the half-edges $j$ and $k$ must be labeled $\overline{B}$ and $B$ respectively, regardless of the configuration.  Each configuration has a valid local move that results in a tree $T''$ with leaf $e(i,i+1)$ as desired. 
\end{proof}

We can use valid plane trees and local moves to create a graph.  Later sections will prove key properties of the graph defined in the next proposition.  The proof follows immediately from earlier propositions.

\begin{proposition}
Fix a primary structure $P$.  Let $\mathcal{G}_P$ be the graph whose vertices are the elements of the set $\mathcal{V}(P)$ with an edge between two  plane trees if they are connected by a valid local move.  Then $\mathcal{G}_P$ is a well-defined undirected graph. 
\end{proposition}
\revised{
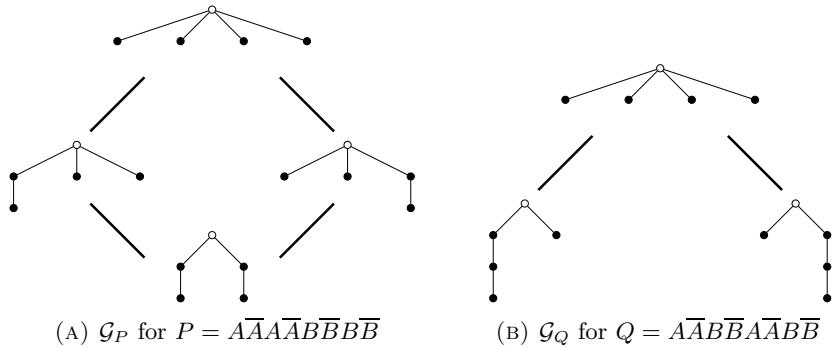
\begin{figure}[H]
\centering
\begin{subfigure}[b]{.35 \textwidth}

\scalebox{.6}{  \begin{tikzpicture}[sloped, level distance=1cm,
level 1/.style={sibling distance=2cm},
level 2/.style={sibling distance=1cm}]
\tikzstyle{every node}=[circle, draw, scale=.8, inner sep=2pt]
\vspace{10pt}
\begin{scope}[scale=.7]
\node[inner sep=2pt] (Root)[] {}
    child {
    node[fill] {}      child {
    node[fill] {}  
    	edge from parent 
	node[ellipse, below, draw=none] {}
	node[ellipse, above, draw=none]  {}  
}
    	edge from parent 
	node[ellipse, below, draw=none] {}
	node[ellipse, above, draw=none]  {}  
}    child {
    node[fill] {} 
    	edge from parent 
	node[ellipse, above, draw=none] {}
	node[ellipse, below, draw=none]  {}  
}
    child {
    node[fill] {}  
    	edge from parent 
	node[ellipse, below, draw=none] {}
	node[ellipse, above, draw=none]  {}  
};
\end{scope}
\begin{scope}[shift={(6,0)}, scale=.7]
\node[inner sep=2pt] (Root)[] {}
    child {
    node[fill] {}  
    	edge from parent 
	node[ellipse, below, draw=none] {}
	node[ellipse, above, draw=none]  {}  
}
    child {
    node[fill] {}  
    	edge from parent 
	node[ellipse, below, draw=none] {}
	node[ellipse, above, draw=none]  {}  
}
    child {
    node[fill] {} 
    child {
    node[fill] {}
    	edge from parent 
	node[ellipse, below, draw=none] {}
	node[ellipse, above, draw=none]  {}  
}
    	edge from parent 
	node[ellipse, above, draw=none] {}
	node[ellipse, below, draw=none]  {}  
}
;
\end{scope}
\begin{scope}[shift={(3,3)}, scale=.7]
\node[inner sep=2pt] (Root)[] {}
    child {
    node[fill] {}  
    	edge from parent 
	node[ellipse, below, draw=none] {}
	node[ellipse, above, draw=none]  {}  
}
    child {
    node[fill] {}  
    	edge from parent 
	node[ellipse, below, draw=none] {}
	node[ellipse, above, draw=none]  {}  
}
    child {
    node[fill] {}  
    	edge from parent 
	node[ellipse, below, draw=none] {}
	node[ellipse, above, draw=none]  {}  
}
    child {
    node[fill] {}  
    	edge from parent 
	node[ellipse, below, draw=none] {}
	node[ellipse, above, draw=none]  {}  
}
;
\end{scope}
\begin{scope}[shift={(3,-2)}, scale=.7]
\node[inner sep=2pt] (Root)[] {}
    child {
    node[fill] {}      child {
    node[fill] {}  
    	edge from parent 
	node[ellipse, below, draw=none] {}
	node[ellipse, above, draw=none]  {}  
}
    	edge from parent 
	node[ellipse, below, draw=none] {}
	node[ellipse, above, draw=none]  {}  
}
   child {
    node[fill] {}      child {
    node[fill] {}  
    	edge from parent 
	node[ellipse, below, draw=none] {}
	node[ellipse, above, draw=none]  {}  
}
    	edge from parent 
	node[ellipse, below, draw=none] {}
	node[ellipse, above, draw=none]  {}  
}
;
\end{scope}
\draw [ultra thick] (.3,.3)--(1.5,1.5);
\draw [ultra thick] (5.7,.3)--(4.5,1.5);
\draw [ultra thick] (.3,-1.3)--(1.5,-2.5);
\draw [ultra thick] (5.7,-1.3)--(4.5,-2.5);
\end{tikzpicture}}
\caption{$\mathcal{G}_P$ for $P=A\overline{A}A\overline{A}B\overline{B}B\overline{B}$}
\end{subfigure}
\begin{subfigure}[b]{.35 \textwidth}
\centering
\scalebox{.6}{
  \begin{tikzpicture}[sloped, level distance=1cm,
level 1/.style={sibling distance=2cm},
level 2/.style={sibling distance=1cm}]
\tikzstyle{every node}=[circle, draw, scale=.8, inner sep=2pt]
\vspace{10pt}
\begin{scope}[scale=.7]
\node[inner sep=2pt] (Root)[] {}
    child {
    node[fill] {} 
    child {
    node[fill] {}
     child {
    node[fill] {}
    	edge from parent 
	node[ellipse, below, draw=none] {}
	node[ellipse, above, draw=none]  {}  
}
   	edge from parent 
	node[ellipse, below, draw=none] {}
	node[ellipse, above, draw=none]  {}  
}
    	edge from parent 
	node[ellipse, above, draw=none] {}
	node[ellipse, below, draw=none]  {}  
}
    child {
    node[fill] {}  
    	edge from parent 
	node[ellipse, below, draw=none] {}
	node[ellipse, above, draw=none]  {}  
};
\end{scope}
\begin{scope}[shift={(6,0)}, scale=.7]
\node[inner sep=2pt] (Root)[] {}
    child {
    node[fill] {}  
    	edge from parent 
	node[ellipse, below, draw=none] {}
	node[ellipse, above, draw=none]  {}  
}
    child {
    node[fill] {} 
    child {
    node[fill] {}
    child {
    node[fill] {}
    	edge from parent 
	node[ellipse, below, draw=none] {}
	node[ellipse, above, draw=none]  {}  
}
    	edge from parent 
	node[ellipse, below, draw=none] {}
	node[ellipse, above, draw=none]  {}  
}
    	edge from parent 
	node[ellipse, above, draw=none] {}
	node[ellipse, below, draw=none]  {}  
}
;
\end{scope}
\begin{scope}[shift={(3,3)}, scale=.7]
\node[inner sep=2pt] (Root)[] {}
    child {
    node[fill] {}  
    	edge from parent 
	node[ellipse, below, draw=none] {}
	node[ellipse, above, draw=none]  {}  
}
    child {
    node[fill] {}  
    	edge from parent 
	node[ellipse, below, draw=none] {}
	node[ellipse, above, draw=none]  {}  
}
    child {
    node[fill] {}  
    	edge from parent 
	node[ellipse, below, draw=none] {}
	node[ellipse, above, draw=none]  {}  
}
    child {
    node[fill] {}  
    	edge from parent 
	node[ellipse, below, draw=none] {}
	node[ellipse, above, draw=none]  {}  
}
;
\end{scope}
\draw [  ultra thick] (.3,.3)--(1.5,1.5);
\draw [  ultra thick] (5.7,.3)--(4.5,1.5);
\end{tikzpicture}}
\caption{$\mathcal{G}_Q$ for $Q=A\overline{A}B\overline{B}A\overline{A}B\overline{B}$}
\end{subfigure}
\caption{Graphs of valid trees.} 
\end{figure}
}

\section{Graph of valid plane trees: Global Structure}

Given a primary {sequence} $P$ this section considers the overall structure of the graph $\mathcal{G}_P$ of $P$-valid plane trees.  We give an algorithm that produces a $P$-valid plane tree from $P$ exactly when the set $\mathcal{V}(P)$ is nonempty.  The algorithm is greedy in the sense that it matches letters in $P$ opportunistically as it reads through the sequence.  Theorem \ref{theorem: graph is connected} shows that the graph $\mathcal{G}_P$ is connected.  Define a graph $\mathcal{G}_P^+$ by directing each edge in $\mathcal{G}_P$ consistent with its type 2 move.   Corollary \ref{cor: unique sink} proves that $\mathcal{G}_P^+$ has a unique sink. This 2-sink is characterized in a number of different combinatorial ways in Corollary \ref{cor: characterize T_0}, including as the tree produced by the greedy algorithm. 

\revised{
\begin{remark}\label{remark: bijection to NCMs}
We streamline proofs in this section by using the well-know bijection between plane trees with $n$ edges and non-crossing matchings on $2n$ letters that sends each edge $e(i,j)$ in the plane tree to the ordered pair $(i,j)$ in the matching.   (The inverse of this map is obtained by writing the half-edges of a plane tree counterclockwise from the root; in each pair $(i,j)$ the index $i$ corresponds to a left-half-edge and the index $j$ corresponds to a right half-edge.) Note that this map respects the poset under inclusion for plane trees and noncrossing matchings, in the sense that subtrees correspond to submatchings and vice versa. The reader interested in more is referred to {\cite[Section 4.4.2]{Bona-Handbook}}.
\end{remark}}

By extension of the classical bijection between plane trees and noncrossing matchings, every $P$-valid plane tree $S$ induces a perfect matching on the letters $p_1p_2\cdots p_{2n}=P$ of the word $P$.  If $p_i$ and $p_j$ label the two half-edges of a single edge in $S$ then we  refer to $p_i$ and $p_j$ as a {\em pair} in the matching. 

We now give
 the algorithm, which we describe in the language of stacks.
\revised{
\begin{definition}[The greedy algorithm]
\label{def: greedy algorithm}
 Given a primary structure $P=p_1p_2\cdots p_{2n}$ perform the following.
 \begin{itemize}
\item Push the first letter $p_1$.
\item For each subsequent letter $p_i$
\begin{itemize}
\item Peek at the stack.  
\item If the letter on top of the stack is the complement of $p_i$ pop it.
\item If not push $p_i$.  
\end{itemize}
\end{itemize}
The greedy algorithm outputs a (possibly empty) stack and also a collection of ordered pairs  
$$E(P)=\{(i,j) : \text{$p_i$ is popped at the $j^{th}$ step of the algorithm} \} .$$ 
\end{definition}}

From the biological perspective, RNA molecules do not wait patiently to be fully formed before starting to fold. As they grow longer RNA molecules begin to refold into shapes that minimize free energy   \cite{beyond-energy-minimization}.  The greedy algorithm models the first folds, before the effects of energy-minimization dominate.

The output of the greedy algorithm satisfies several properties that are reminiscent of Proposition \ref{proposition: properties of plane trees} and Proposition \ref{proposition: valid subtrees}.  \rerevised{It creates a partial matching on each input sequence $P$.  In fact, if the algorithm terminates with an empty stack then it created a perfect matching which is moreover a $P$-valid plane tree denoted $T_0(P)$.  The following lemma is the main step in our proof that the greedy algorithm always produces a valid plane tree if any $P$-valid plane trees exist.}
\revised{
\begin{figure}[H]
\centering
\begin{subfigure}[b]{0.4\textwidth}

\begin{tikzpicture}[level distance=1cm,
level 1/.style={sibling distance=2cm},
level 2/.style={sibling distance=1cm}]
\tikzstyle{every node}=[circle, draw, scale=.8, inner sep=2pt]
\vspace{10pt}
\node (Root)[] {}
    child {
    node[fill] {}
    child { node[fill] {} 
	edge from parent 
	node[left, draw=none] {$B$}
	node[right, draw=none]  {$\overline{B}$}    
    }
	edge from parent 
	node[above, draw=none] {$B$}
	node[below, draw=none]  {$\overline{B}$}
}
    child{ node[fill] {} 
	edge from parent 
	node[left, draw=none] {${A}$}
	node[right, draw=none]  {$\overline{A}$}    
    }
    child{ node[fill] {} 
	edge from parent 
	node[below, draw=none] {$B$}
	node[above, draw=none]  {$\overline{B}$}    
    };
\end{tikzpicture}
\end{subfigure}
\begin{subfigure}[b]{0.4\textwidth}
\centering
\begin{tikzpicture}[level distance=1cm,
level 1/.style={sibling distance=2cm},
level 2/.style={sibling distance=2cm}]
\tikzstyle{every node}=[circle, draw, scale=.8, inner sep=2pt]
\vspace{10pt}
\node (Root)[] {}
    child {
    node[fill] {}
    child { node[fill] {} 
	edge from parent 
	node[above, draw=none] {$B$}
	node[below, draw=none]  {$\overline{B}$}    
    }
    child{ 
       node[fill] {} 
 child{ node[fill] {} 
	edge from parent 
	node[left, draw=none] {${A}$}
	node[right, draw=none]  {$\overline{A}$}    
    }
	edge from parent 
	node[above, draw=none] {$B$}
	node[below, draw=none]  {$\overline{B}$}    
}
	edge from parent 
	node[left, draw=none] {$B$}
	node[right, draw=none]  {$\overline{B}$}};
\end{tikzpicture}
\end{subfigure}
\caption{Both plane trees are valid for $P=BB\overline{B} \overline{B}A \overline{A}B \overline{B}$ but the one on the left is $T_0(P)$.}
\end{figure}
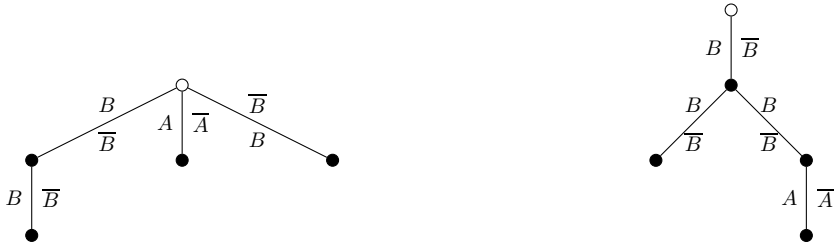}

\rerevised{\begin{lemma}\label{lemma: properties of greedy output}
Suppose that when the greedy algorithm is applied to the sequence  $P=p_1p_2\cdots p_{2n}$ it pops $p_i$ at the $j^{th}$ step.
Let $P'$ be the sequence obtained by removing the letters ${\color{black}p_i}p_{i
+1}p_{i+2}\cdots p_{j-1}{\color{black}p_j}$ from $P$ and retaining (for the sake of notational convenience) the same indexing.  Then implementing the greedy algorithm on $P'$ results in the same stack as $P$ and moreover $E(P)$ is the disjoint union
\[E(P) = E(P') \bigcup M\]
where $M$ is a perfect matching on the set $\{{\color{black}i,}i+1,i+2,\ldots,j-1{\color{black},j}\}$.
\end{lemma}}

\begin{proof}
If $p_i$ is popped at the $j^{th}$ step then it must be at the top of the stack.  Thus any of the letters $p_{i+1},p_{i+2},\ldots,p_{j-1}$ that were pushed must also have been popped before the $j^{th}$ step.  Since every symbol is popped at most once, none of $p_1, p_2, \ldots, p_{i-1}$ were popped in steps $i+1,i+2,\ldots,j-1$.  This means the subset of $p_{i+1},p_{i+2},\ldots,p_{j-1}$ that were pushed onto the stack is bijective with the number of steps in $\{i+1,i+2,\ldots,j-1\}$ at which the stack was popped, or equivalently $E(P)$ contains a perfect matching of $\{i+1,i+2,\ldots,j-1\}$.  It follows that the stack at the $j+1^{th}$ step on input $P$ is the same as the stack at the $i-1^{th}$ step of input $P$.  Since these are the same as the stack at the $i-1^{th}$ step of $P'$ and since $P$ and $P'$ share the same $2n-j$ last letters, the greedy algorithm produces the same output (pairs and stack) on $P'$ as on $P$ outside of $\{{\color{black}i, }i+1,i+2,\ldots,j-1{\color{black}, j}\}$.
%
\end{proof}

Note that $i$ labels a left half-edge in the tree $T_0(P)$ if and only if $p_i$ was pushed onto the stack in the greedy algorithm.

In fact if a primary structure has any valid plane tree, then the greedy algorithm will terminate with an empty stack and thus produce a valid plane tree.  The next theorem proves this claim using {a proof by minimal counterexample.}

\revised{
\begin{theorem}
Let $P$ be a primary structure.  If $\mathcal{V}(P)$ is nonempty then the greedy algorithm produces a valid plane tree $T_0\in \mathcal{V}(P)$. 
\end{theorem}
\begin{proof}
If $n=1$ then for all sequences of length $2n=2$ that have a valid plane tree, the greedy algorithm produces the (unique) plane tree.

Consider the collection of primary sequences $P$ such that a $P$-valid plane tree exists but the greedy algorithm terminates with a nonempty stack.  Let $Q=q_1q_2\cdots q_{2k}$ be such a sequence of shortest length, i.e. $Q$ is a sequence of length $2k$ and for all valid sequences of length $2\ell$ where $\ell<k$ the greedy algorithm terminates with an empty stack.

We first claim that the greedy algorithm pops at least once.  Indeed there is at least one $Q$-valid plane tree and that tree must have at least one leaf, say edge $e(i,i+1)$.  This means that $q_i = \overline{q_{i+1}}$ and so the greedy algorithm pops at the $i+1^{th}$ step if not before. 

We use this fact to produce a shorter sequence than $Q$ for which a valid plane tree exists but the greedy algorithm terminates without an empty stack.

Suppose that the algorithm pops first at the $j+1^{th}$ step.  Then it must have popped $q_j$ off the stack.  By Lemma \ref{lemma: properties of greedy output} the greedy algorithm has the same stack when run on the sequence 
\[Q' = q_1q_2 \cdots q_{j-2}q_{j-1}q_{j+2}q_{j+3}\cdots q_{2n}\] 
as on $Q$.  If $e(j,j+1)$ is a leaf in one of the Q-valid plane trees $T$ then plucking $e(j,j+1)$ off of $T$ results in a $Q'$-valid plane tree.  This \rerevised{would then contradict} our hypothesis that $Q$ was the minimal such sequence.

It remains to be shown that $e(j,j+1)$ is indeed a leaf in one of the $Q$-valid plane trees.  Suppose not and consider an arbitrary $T\in \mathcal{V}(Q)$.  Consider the edges containing the $j$ and $j+1^{th}$ half-edges.  They are adjacent edges that have complementary labels since the $j+1^{th}$ step of the greedy algorithm pops $q_j$.  By either a type 1 or type 2 move, the tree $T$ can be transformed into a new valid plane tree $T'$ in $\mathcal{V}(Q)$ containing the leaf $e(j,j+1)$ as desired.
\end{proof}}

We can say more about the valid plane tree that the greedy algorithm produces.

\begin{corollary}\label{corollary: $T_0$ has no type 2}
The greedy algorithm produces a valid plane tree with no valid local moves of type 2.
\end{corollary}

\begin{proof}

Consider a valid local move of type 2 in $T_0$. This local move involves two adjacent edges $e(i,j')$ and $e(j,i')$ in $T_0$ such that $i<j<i'<j'$. If the half-edge $i$ has label $B$ then the half-edges $j$ and $j'$ must both be labeled by $\overline{B}$ while half-edge $i'$ is labeled $B$.  At the $j^{th}$ step of the greedy algorithm $p_{i}$ is on top of the stack (regardless of whether there are other edges between $i$ and $j$ by Proposition \ref{proposition: valid subtrees}).  So $p_{i}$ must be popped forming $e(i,j)$. By contradiction we conclude $T_0$ has no valid local moves of type 2.
\end{proof}

In fact more is true: the tree $T_0$ produced by the greedy algorithm is the only $P$-valid plane tree  with no valid local moves of type 2.  The proof walks through the definitions, particularly using the non-crossing condition of half-edges in plane trees.


\begin{theorem}\label{theorem: non-greedy trees have type 2}
Suppose $T$ is a $P$-valid plane tree that is not $T_0$.  Then $T$ has a valid local move of type 2.
\end{theorem}
\revised{
\begin{proof}
\rerevised{Consider a $P$-valid plane tree $T$ that is not $T_0$. They differ in some edge, so let $j$ denote the first half-edge in which they differ, i.e. $p_j$ labels a left half-edge in one of $T$ or $T_0$ but a right half-edge in the other.}  We first show that $p_j$ must label a right half-edge in $T_0$.  If not it labels a right half-edge in $T$ which we call $e(i',j)$.  The half-edges $i'+1,\ldots,j-1$ form a subtree of $T$ by Proposition \ref{proposition: valid subtrees}.  This same subtree is in both $T$ and $T_0$ by hypothesis that $p_j$ is the first label where they differ.  Thus $p_{i'}$ is on top of the stack at step $j$ in the greedy algorithm and so $e(i',j)$ is in $T_0$ contradicting our hypothesis.

So $p_j$ labels a right half-edge in $T_0$ say of the edge $e(i,j)$.  Then
\begin{itemize}
\item $i$ and $j$ are both left half-edges in $T$ by hypothesis
\item that have complementary labels, say $B$ and $\overline{B}$ respectively, because $T_0$ is $P$-valid and
\item the half-edges labeled ${i+1}, {i+2}, \ldots, {j-1}$ form a subtree in $T_0$ by Proposition \ref{proposition: valid subtrees} and thus in $T$ by hypothesis.
\end{itemize}
\begin{figure}[H]
\begin{center}
\centering
  \begin{tikzpicture}[level distance=2cm,
level 1/.style={sibling distance=2cm},
level 2/.style={sibling distance=1cm}]
\node [left, ellipse, draw, dashed] at (0,-2) {\begin{tabular}{c} matches $T_0$\end{tabular}};
\node [ellipse, draw, dashed] at (0,.5) {left half-edges match $T_0$};
\tikzstyle{every node}=[circle, draw, scale=.8, inner sep=2pt]
\vspace{10pt}
\node (Root)[fill] {}
    child {
    node[fill] {}
    child { node[fill] {} 
	edge from parent 
	node[left, draw=none] {$p_j=B$}
	node[right, draw=none]  {$\overline{B}$}    
    }
	edge from parent 
	node[left, draw=none] {$p_{i}=\overline{B}$}
	node[right, draw=none]  {$B$}
};

     \node at (0,1){};
     \draw [dashed] (.5,-2) circle [radius=0.5];
          \draw [dashed] (0,-4.5) circle [radius=0.5];
\end{tikzpicture}
\end{center}
\caption{If $T$ differs from $T_0$ then $T$ has a type-$2$ local move.}\label{figure: have local move}
\end{figure}
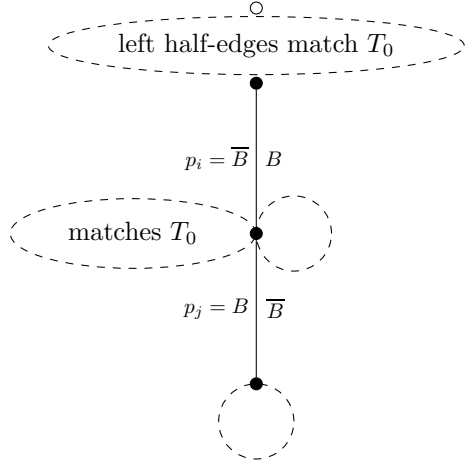
This means that $T$ looks like the schematic of Figure \ref{figure: have local move} and so admits a local move of type 2, as desired.
\end{proof}
}
We now conclude that in fact the graph whose vertices are valid plane trees and whose edges are local moves is connected.

\begin{theorem}\label{theorem: graph is connected}
Fix a primary structure $P$.  The graph $\mathcal{G}_P$ is connected.
\end{theorem}

\begin{proof} 
We induct on the {length} of the sequence $P$.  The claim is trivial when $P$ has $2$ letters since in that case there is a unique valid plane tree.  Suppose that for every primary structure with at most $2(n-1)$ letters, the graph  whose vertices are valid plane trees and whose edges are local moves is connected.  Fix a primary structure $P$ with $n$ letters and consider two distinct $P$-valid plane trees $S$ and $T$.  If $S$ and $T$ share no common edges then Proposition \ref{proposition: sharing an edge} asserts that there is a valid local move on one to a $P$-valid plane tree sharing an edge with the other.  Without loss of generality assume that $S$ and $T$ both contain the edge $e(i,j)$.

Now delete the common edge $e(i,j)$.  Create the subtrees $S'$ and $T'$ induced from $S$ and $T$ respectively from the edges farther from the root than $e(i,j)$ (namely those indexed $e(i',j')$ for $i<i', j'<j$), and create subtrees $S''$ and $T''$ respectively from the edges closer to the root than $e(i,j)$ (namely those indexed $e(i',j')$ for $i'<i$ and $j'>j$).  By Proposition \ref{proposition: valid subtrees} we know $S'$ and $T'$ are $P'$-valid plane trees for the subsequence $P'$ of $P$ consisting of the letters in positions $i'$ for $i'<i$ or $i'>j$.  Similarly $S''$ and $T''$ are $P''$-valid plane trees, where $P''$ is the subsequence of $P$ consisting of the letters in positions $i'$ for $i<i'<j$.  By induction there are valid local moves taking $S'$ to $T'$ and $S''$ to $T''$.  Implement the same local moves on $S$ to get valid local moves that take $S$ to $T$.  This proves the claim.
\end{proof}

{Furthermore} if the edges of $\mathcal{G}_P$ are directed and correspond to type $2$ local moves, then the graph has a unique sink $T_0$.

\begin{corollary}
\label{cor: unique sink}
Fix a primary sequence $P$ with {at least one} valid plane tree and let $T_0$ be the valid plane tree for $P$ produced by the greedy algorithm.  Let $\mathcal{G}_P^+$ denote the graph whose vertices are valid plane trees for $P$ and with a directed edge $T \rightarrow T'$ if there is a local move of type 2 from $T$ to $T'$.

The valid plane tree $T_0$ is the unique sink in the directed graph $\mathcal{G}_P^+$.
\end{corollary}

\begin{proof}
Corollary \ref{corollary: $T_0$ has no type 2} proved that $T_0$ has no edges directed out  so $T_0$ is a sink in $\mathcal{G}_P^+$.  Theorem \ref{theorem: non-greedy trees have type 2} proved that every other tree has at least one edge directed out so there are no other sinks in the graph $\mathcal{G}_P^+$.  
\end{proof}
\begin{remark}
The graph $\mathcal{G}_P^+$ may have several sources.  For instance  the graph on the sequence $A\overline{A}B\overline{B}A\overline{A}B\overline{B}$ has two sources and one sink. 
\begin{figure}[H]
\centering
\scalebox{.7}{
  \begin{tikzpicture}[sloped, level distance=1cm,
level 1/.style={sibling distance=2cm},
level 2/.style={sibling distance=1cm}]
\tikzstyle{every node}=[circle, draw, scale=.8, inner sep=2pt]
\vspace{10pt}
\begin{scope}[scale=.7]
\node[inner sep=2pt] (Root)[] {}
    child {
    node[fill] {} 
    child {
    node[fill] {}
     child {
    node[fill] {}
    	edge from parent 
	node[ellipse, below, draw=none] {}
	node[ellipse, above, draw=none]  {}  
}
   	edge from parent 
	node[ellipse, below, draw=none] {}
	node[ellipse, above, draw=none]  {}  
}
    	edge from parent 
	node[ellipse, above, draw=none] {}
	node[ellipse, below, draw=none]  {}  
}
    child {
    node[fill] {}  
    	edge from parent 
	node[ellipse, below, draw=none] {}
	node[ellipse, above, draw=none]  {}  
};
\end{scope}
\begin{scope}[shift={(6,0)}, scale=.7]
\node[inner sep=2pt] (Root)[] {}
    child {
    node[fill] {}  
    	edge from parent 
	node[ellipse, below, draw=none] {}
	node[ellipse, above, draw=none]  {}  
}
    child {
    node[fill] {} 
    child {
    node[fill] {}
    child {
    node[fill] {}
    	edge from parent 
	node[ellipse, below, draw=none] {}
	node[ellipse, above, draw=none]  {}  
}
    	edge from parent 
	node[ellipse, below, draw=none] {}
	node[ellipse, above, draw=none]  {}  
}
    	edge from parent 
	node[ellipse, above, draw=none] {}
	node[ellipse, below, draw=none]  {}  
}
;
\end{scope}
\begin{scope}[shift={(3,3)}, scale=.7]
\node[inner sep=2pt] (Root)[] {}
    child {
    node[fill] {}  
    	edge from parent 
	node[ellipse, below, draw=none] {}
	node[ellipse, above, draw=none]  {}  
}
    child {
    node[fill] {}  
    	edge from parent 
	node[ellipse, below, draw=none] {}
	node[ellipse, above, draw=none]  {}  
}
    child {
    node[fill] {}  
    	edge from parent 
	node[ellipse, below, draw=none] {}
	node[ellipse, above, draw=none]  {}  
}
    child {
    node[fill] {}  
    	edge from parent 
	node[ellipse, below, draw=none] {}
	node[ellipse, above, draw=none]  {}  
}
;
\end{scope}
\draw [->,  ultra thick] (.3,.3)--(1.5,1.5);
\draw [->,  ultra thick] (5.7,.3)--(4.5,1.5);
\end{tikzpicture}}
\caption{$\mathcal{G}_Q^+$ for $Q=A\overline{A}B\overline{B}A\overline{A}B\overline{B}$}
\label{fig: graph with one sink two sources}
\end{figure}
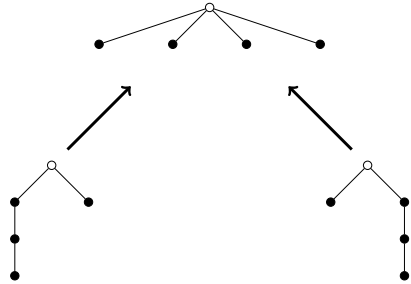
\end{remark}

\begin{corollary}
The graph $\mathcal{G}_P^+$ has no directed cycles.  There is a path consisting only of type-2 valid local moves from every valid plane tree $T$ in $\mathcal{G}_P^+$ to $T_0$.
\end{corollary}

\begin{proof}
Consider the total distance from the root in a valid plane tree, defined as 
\[d_T = \sum_{v \in T} dist(v,v_0)\] 
where $v_0$ is the root.  Figure \ref{fig: type 2 dist} shows that $d_T$ drops by at least one with each valid local move of type 2 (and more if the subtree labeled C in Figure \ref{fig: type 2 dist} is non-empty).
\begin{figure}[H]
\centering
\begin{tikzpicture}[sloped, level distance=1cm,
level 1/.style={sibling distance=2cm},
level 2/.style={sibling distance=1cm}]
\node at (0,.5) {A};
\node at (-1,-1.5) {B};
\node at (0,-.65) {C};
\node at (1,-1.5) {D};
\begin{scope}[shift={(-5,0)}]
\node at (0,.5) {A};
\node at (-.5,-1) {B};
\node at (0,-2.5) {C};
\node at (.5,-1) {D};
\end{scope}
\draw [->, ultra thick] (-3,-1)--(-2,-1);
\node [above] at (-2.5, -1) {type 2};
\tikzstyle{every node}=[circle, draw, scale=.8, inner sep=2pt]
\node[inner sep=2pt] (Root)[fill] {}
    child {
    node[fill] {}  
    	edge from parent 
	node[ellipse, below, draw=none] {}
	node[ellipse, above, draw=none]  {}  
}
    child {
    node[fill] {}  
    	edge from parent 
	node[ellipse, below, draw=none] {}
	node[ellipse, above, draw=none]  {}  
}
;
\draw [dashed] (0,0.5) circle [radius=0.5];
\draw [dashed] (-1,-1.5) circle [radius=0.5];
\draw [dashed] (1,-1.5) circle [radius=0.5];
\draw [dashed] (0,-.65) circle [radius=0.5];
\node at (0,1){};
\begin{scope}[shift={(-5,0)}]
\node[inner sep=2pt] (Root)[fill] {}
    child {
    node[fill] {}      child {
    node[fill] {}  
    	edge from parent 
	node[ellipse, below, draw=none] {}
	node[ellipse, above, draw=none]  {}  
}
    	edge from parent 
	node[ellipse, below, draw=none] {}
	node[ellipse, above, draw=none]  {}  
};
\draw [dashed] (0,0.5) circle [radius=0.5];
\draw [dashed] (-.5,-1) circle [radius=0.5];
\draw [dashed] (.5,-1) circle [radius=0.5];
\draw [dashed] (0,-2.5) circle [radius=0.5];
\node at (0,1){};
\end{scope}
\end{tikzpicture}
\caption{A generic move of type 2}
\label{fig: type 2 dist}
\end{figure}
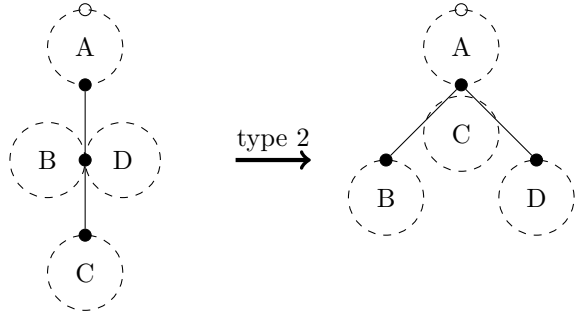
Since the total distance $d_T$ can only decrease there are no directed cycles in the graph $\mathcal{G}_P^+$. 

Moreover Theorem \ref{theorem: graph is connected} proved that the undirected graph $\mathcal{G}_P$ is connected and hence $\mathcal{G}_P^+$ has only one connected component.  Consider the following process: given a tree $T$ in the graph $\mathcal{G}_P^+$ follow any edge out of $T$.  Repeat until you reach a tree with no edges directed out within the graph $\mathcal{G}_P^+$.  Since $\mathcal{G}_P^+$ has no directed cycles and only finite number of trees, we know this process terminates.  Since $T_0$ is the unique tree with no edges directed out of it in $\mathcal{G}_P^+$ we know the process terminates at $T_0$. Thus there is a path consisting only of type 2 valid local moves from every valid plane tree $T$ to $T_0$ as desired. 
\end{proof}

We collect all of this information concisely as follows.

\begin{corollary}
\label{cor: characterize T_0}
Let $P$ be a primary sequence with at least one $P$-valid plane tree.  The following are equivalent.
\begin{itemize}
\item $T_0$ is the unique sink in the graph $\mathcal{G}_P^+$.
\item $T_0$ is the plane tree produced by the greedy algorithm on $P$.
\item For each plane tree $T$ with root $v_0$ define the {\em total distance} $d_T= \sum_{v \in T} dist(v,v_0)$.  Among the $P$-valid plane trees $T_0$ has minimum total distance.
\end{itemize}
\end{corollary}

\section{Number of Valid Plane Trees: the Size of $\mathcal{G}_P$}
\label{sec: number of valid trees}

A primary sequence $P$ has two main parameters: the length of the sequence $P$ and the size of the alphabet from which the letters of $P$ are drawn.  The vast majority of primary sequences have no valid plane trees at all, as we discuss in the next section.  This section gives results showing how the order of $\mathcal{G}_P$ depends on $n$ and $m$.  It also includes open questions.

Some notation makes the discussion in this and later sections easier. 
\begin{notation} Given sequences of length $2n$ in complementary alphabets with $2m$ letters:
\begin{itemize}
\item $\mathcal{P}(n,m)$ is the set of valid sequences of length $2n$ in a complementary alphabet with $2m$ letters, namely sequences $P$ with at least one $P$-valid plane tree.
\item $N(n,m,k)$ is the set of valid sequences $P\in \mathcal{P}(n,m)$ such that $|\mathcal{V}(P)|=k$.
\item $R(n,m)$ is the set of $k \in \mathbb{Z}$ such that $N(n,m,k)$ is non-empty.
\end{itemize}
\end{notation}


\subsection{Describing $N(n,m,k)$}
\revised{
We can bound the $k$ for which $N(n,m,k)$ is nonzero and can describe $N(n,m,k)$ in the boundary cases.  We have the following initial observations, which are generalizations of observations made by Kemp, Mahlburg, Rattan, and Smyth \cite[Proposition 1.7]{KMRS}.
\begin{remark} \label{remark: bounding V(P)} A sequence obtains the maximum number of valid plane trees in the following circumstances: 
\begin{itemize}
\item A primary sequence $P$ of length $2n$ has at most $C_n$ valid plane trees, where $C_n$ is the $n^{th}$ Catalan number.
\item A primary sequence $P$ has exactly $C_n$ valid plane trees if and only if $P$ is of the form $B \overline{B} B \overline{B}...B \overline{B}$. 
\item Using an alphabet of size $2m$ there are $2m$ such sequences, namely $|N(n,m, C_n)|=2m$.
\end{itemize}
\end{remark}}

If a sequence in $\mathcal{P}(n,m)$ has more than $C_{n-1}$ valid plane trees, then all $C_n$ possible plane trees must be valid.  In other words $R(n,m)$ has a large gap between its largest element $C_n$ and its second largest $C_{n-1}$.  Our argument uses the recurrence relation for Catalan numbers on a sequence that is not $(B\overline{B})^n$. 

\revised{
\begin{proposition}\label{proposition: nothing between C_n and C_{n-1}}
\rerevised{Fix $k >0$.} If a sequence $P$ of length $2n$ has $k$ valid plane trees and $k\neq C_n$ then $k$ is at most $C_{n-1}$.
\end{proposition}

\begin{proof}
Let $P \in \mathcal{P}(n,m)$ have exactly $k$ valid plane trees with $k< C_n$.  For each $p_i$ in $P$   the half-edge $i$ can only be paired with $j$ if $p_j=\overline{p_i}$ and $j$ is of the opposite parity to $i$. Suppose $P$ contains the letter $B$ in exactly $\ell_{B}>0$ odd-indexed places. \rerevised{Since there is at least one $P$-valid plane tree, there must be $\ell_{B}$ even-indexed places containing the letter $\bar{B}$. In any valid tree $T\in \mathcal{V}(P)$ these $\ell_B$ edges can, if adjacent, interact via local moves and thus
 form at most $C_{\ell_B}$ valid subforests.} This is true for any letter in $\mathcal{A}$ so the set $\mathcal{V}(P)$ has at most $k= \prod \limits_{B\in \mathcal{A}} C_{\ell_B}$ valid trees.  

Let $B_1$ denote the letter $p_1$. The product $C_a\cdot C_b$ is less than or equal to $C_{a+b}$ for any $a,b \geq 0$ by the recurrence relation defining Catalan numbers. Moreover each $\ell_B \leq n-1$ since $P \neq (B \overline{B})^n$ by hypothesis.  This means
$$k \leq \prod \limits_{B\in \mathcal{A}} C_{\ell_B} \leq  C_{\ell_{B_1}} \cdot C_{n-\ell_{B_1}}$$
The product of Catalan numbers $C_i \cdot C_{n-i}$ is maximized when $i=1$ or $i=n-1$.  So 
$$ C_{\ell_{B_1}} \cdot C_{n-\ell_{B_1}} \leq C_1\cdot C_{n-1}=C_{n-1}.$$
\end{proof}

We can also characterize the sequences $P$ that achieve the bound of Proposition \ref{proposition: nothing between C_n and C_{n-1}}.
\begin{corollary}
If a sequence $P = p_1p_2\cdots p_{2n}$ has exactly $C_{n-1}$ valid plane trees then
\begin{itemize}
\item all but one of the odd indexed $p_i$ are the same letter,
\item all but one of the even indexed $p_i$ are the same letter, and
\item if for $i$ and $j$ of opposite parity $p_i$ and $p_j$ differ from the other odd and even indexed letters respectively, then either
\begin{itemize}
\item one of $i,j$ is $1$ and the other is $2n$, or
\item $|i-j|=1$.
\end{itemize}
\end{itemize}
\end{corollary}
\begin{proof}
By the previous proof if more than two letters in $\mathcal{A}$ appear as $p_i$ for $i$ odd then $|\mathcal{V}(P)|$ will be strictly less than $C_{n-1}$.  If exactly two letters appear in odd-indexed positions, say $A$ and $B$ (possibly complements), then there are at most $C_{\ell_A}\cdot C_{\ell_B}$ $P$-valid plane trees. By hypothesis there are $C_{n-1}$ total $P$-valid trees so without loss of generality $\ell_A=n-1$ and $\ell_B=1$.  Thus all but one odd-indexed $p_i$ are $A$ and because $P$ is valid, all but one even-indexed $p_i$ are $\overline{A}$.

\rerevised{Let $p_i$ and $p_j$ be the odd-indexed letter that differs from $A$ and the even-indexed letter that differs from $\overline{A}$.  We assume that $i < j$ though not that $i$ is the odd index.} Then any $P$-valid plane tree contains the edge $e(i,j)$ according to the following schematic:
\begin{figure}[H]
\centering
 \begin{tikzpicture}[level distance=2cm,
level 1/.style={sibling distance=2cm},
level 2/.style={sibling distance=1cm}]
\node [below, ellipse, draw, dashed] at (0,-2) {\begin{tabular}{c} $a$ edges \end{tabular}};
\node [above, ellipse, draw, dashed] at (0,0) {\begin{tabular}{c} $n-a-1$ edges \end{tabular}};
\tikzstyle{every node}=[circle, draw, scale=.8, inner sep=2pt]
\vspace{10pt}
\node (Root)[fill] {}
    child {
    node[fill] {}
	edge from parent 
	node[left, draw=none] {$i$}
	node[right, draw=none]  {$j$}
};

     \node at (0,1){};
\end{tikzpicture}
\caption{Edge $e(i,j)$ partitions a valid plane tree into valid subtrees.} \label{figure: edge partitioning tree}
\end{figure}
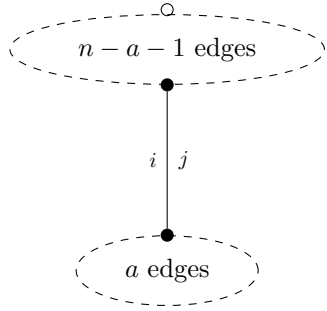
There are $C_a\cdot C_{n-a-1} = C_{n-1}$ possible $P$-valid plane trees 
so either $a=0$ and $e(i,j)$ is a leaf or $a=n-1$ and $e(i,j)$ is the unique edge incident to the root, as claimed.
\end{proof}}

The next corollary follows immediately by counting possible locations and letters for $A$ and $B$.
\begin{corollary}
 There are $2m(2m-1)(2n)$ sequences $P\in \mathcal{P}(n,m)$ with exactly $C_{n-1}$ valid plane trees, i.e. 
\[|N(n,m,C_{n-1})|=4m(2m-1)n.\]
\end{corollary}

We can  characterize the set $N(n,1,1)$ using a similar argument. 

\begin{proposition}
The set $N(n,1,1)$ has size $2n$.
\end{proposition}
\begin{proof}
We show that unless the primary sequence $P$ on the alphabet $\{B, \overline{B}\}$ has a very special form, then we can find a $P$-valid plane tree with a valid local move.  This would imply $\mathcal{V}(P)$ has more than one element and thus $P \not \in N(n,1,1)$.

Suppose $v$ is a vertex in a $P$-valid plane tree on the alphabet $\{B, \overline{B}\}$ that is incident to three edges.  Then two of the left-edges incident to $v$ have the same label by the pigeonhole principle.  Thus there is a valid local move at the vertex $v$.  It follows that if $P$ has only one $P$-valid plane tree then that tree is a path.  The root may be anywhere on that path.  To prevent local moves, all edges to the left of the root must have the same letter on their left half-edges while edges to the right must have its complement on their left half-edges. 
There are two choices for the side that has $B$ and $n$ choices for the location of the root.  Thus there are $2n$ primary sequences $P$ with $|\mathcal{V}(P)|=1$.
\end{proof}

Without loss of generality, all such sequences $P$ must  have the form $P=B^k(\overline{B})^n B^{n-k}$. \rerevised{Heitsch and Poznanovi{\'c} used similar arguments to prove that if $T$ is a plane tree in which the maximum degree of a vertex is $d$, then there exists a sequence $P \in N(n, m,1)$ with $T$ as its unique valid tree if and only if $m \geq \lceil \frac{d}{2} \rceil$ \cite{DiscreteTopologicalModels}. These arguments lead us in two directions.}
\revised{
\begin{question} Consider $N(n,m,1)$ for $m>1$.
 The previous argument can be modified to show that any plane tree with a vertex of degree at least $2m+1$ cannot represent a primary sequence for $P \in N(n,m,1)$. 
\begin{itemize}
\item Can we characterize the primary sequences $P \in N(n,m,1)$?  
\item Given a plane tree $T$  can we characterize the sequences $P\in N(n,m,1)$ such that $\mathcal{V}(P)=\{T\}$? 
\end{itemize}
\rerevised{A related question, about the total number of possible secondary structures, can be addressed using Motzkin numbers, as described in \cite{DiscreteTopologicalModels}. However that enumeration does not count the number of sequences that have only one secondary structure.}

\end{question}
}

\subsection{Comparing $R(n,m)$ for different $n$ and $m$}

Proposition \ref{proposition: nothing between C_n and C_{n-1}} can be interpreted as showing that $|R(n,m)| \leq C_{n-1}+1$.  We further analyze the sets $R(n,m)$ and exploit prime factorization to construct sets $R(n,m)$ that contain desired elements.

\begin{proposition}
Let $k\in \mathbb{Z}_{\geq 0}$ factor into a product of primes $k=k_1^{\alpha_1}\cdots k_\ell^{\alpha_\ell}$ such that for each $i$, there exists a pair $(n_i,m_i)$ with $k_i \in R(n_i,m_i)$.  Then 
\[k \in R \left( \sum \limits_{i=1}^\ell \alpha_in_i, \sum \limits_{i=1}^\ell \alpha_im_i  \right)\]
\end{proposition}
\begin{proof}
Fix a complementary alphabet with $2\sum_{i=1}^{\ell} \alpha_i m_i$ letters.  For each $i$ with $1 \leq i \leq \ell$ let $P_i$ be a primary sequence in $N(n_i, m_i, k_i)$.  Such a $P_i$ exists since $k_i \in R(n_i,m_i)$. For each $i$ make $\alpha_i$ copies of $P_i$ and denote them $P_{i,1}, P_{i,2},\ldots,P_{i,\alpha_i}$.  Now change the complementary letters as needed so that no two primary sequences $P_{i,j}, P_{i',j'}$ share a letter  \rerevised{unless $i=i'$ and $j=j'$. } 

Consider the sequence $P=P_{1,1}P_{1,2}P_{1,3}\cdots P_{\ell,\alpha_{\ell}}$ obtained by concatenating these sequences. No half-edge from $P_{i,j}$ can match with a half-edge from $P_{i',j'}$ unless $i=i'$ and $j=j'$ since the alphabets used for the subsequences $P_{i,j}$ and $P_{i',j'}$ are distinct.  Thus every $P$-valid plane tree consists of $P_{i,j}$-valid plane subtrees joined at a common root and the only $P$-valid local moves are $P_{i,j}$-valid local moves on the subtrees.  Since there are $k_i$ of the $P_{i,j}$-valid plane trees and since the $P_{i,j}$-valid local moves are independent, there are $k=k_1^{\alpha_1}\cdots k_\ell^{\alpha_\ell}$ trees in $\mathcal{V}(P)$ as desired.
\end{proof}

Wagner proved a conjecture of an earlier version of this manuscript that for all $k \in \mathbb{Z}_{\geq 0}$ there exists $n,m$ such that $k\in R(n,m)$.  In fact, he showed the stronger result that every integer $k$ appears in $R(n,1)$.  He also proved that $R(n,m)$ is not generally the same set as $R(n,1)$ which disproved another conjecture of ours \cite{Wagner}.

\section{How many sequences are valid?
} \label{section: number of valid sequences}

This last section enumerates the total number of primary sequences that have {\em any} valid plane trees. These sequences are extremely unusual, in the sense that the ratio of sequences that are valid in a fixed alphabet approaches zero as sequence length increases.  

Our main tool comes from Corollary \ref{cor: characterize T_0}: every primary sequence $P$ corresponds to at most one labeled tree  with no valid local moves of type 2, namely the tree {$T_0$} produced by the greedy algorithm.  Thus we may count the number of valid primary sequences by counting the number of labeled plane trees with no type $2$ local moves. 

\begin{lemma} \label{lemma: every greedy output attained}
Every plane tree $T$ is  the output $T_0(P)$ of the greedy algorithm for some valid sequence $P$.
\end{lemma}
\begin{proof}
Given a plane tree $T$, label the edges incident to the root with any complementary pair from the alphabet.  Each edge not incident to the root has a parent edge.  Starting from the edges adjacent to the root, label each edge so that it has no local move with its parent edge.  Then the labeled tree $T$ has no local moves of type 2 and thus is the output $T_0(P)$ of the greedy algorithm for the sequence $P$ obtained by reading the labels on the half-edges. 
\end{proof}
Note that a labeled tree generated by this procedure does not produce a tree with {\it no} local moves, only with none of type 2.  For example if a plane tree has a vertex of degree $2m+1$ or more, there is always a local move; the process in the previous lemma simply ensures that it is a local move of type 1.  

We can enumerate valid sequences by counting labeled trees with no local moves of type 2.
\begin{theorem}
The number of valid primary sequences is
\begin{equation}
\label{eq: valid seq non recur}
|\mathcal{P}(n,m)|= \sum \limits_{\substack{\lambda=(\lambda_1,...,\lambda_k) \\ \text{a composition of $n$}}} (2m)^k\cdot (2m-1)^{n-k}\cdot \prod \limits_{i=1}^k C_{\lambda_i-1}  
\end{equation}
where $C_{\lambda_i-1}$ is the $(\lambda_i-1)^{th}$ Catalan number.
\end{theorem}
\begin{proof}
By Lemma \ref{lemma: every greedy output attained} we can count the number of valid sequences by counting the number of plane trees with no local moves of type 2.  We partition the set of plane trees based on the number of edges incident to the root and the size of the subtrees coming off of those edges.

Suppose there are $k$ edges incident to the root.  The $k$ subtrees descending from those $k$ edges divide the $n$ edges of the tree into $k$ parts, say with the $i^{th}$ of size $\lambda_i$ including the edge incident to the root.  This gives a composition $\lambda=(\lambda_1,...,\lambda_k)$ of $n$.  Moreover the  $i^{th}$ subtree can be any of $C_{\lambda_i-1}$ possible plane trees. 

Label the edges in order of their distance to the root.  Each of the $k$ edges incident to the root has no parent edge and thus has $2m$ possible labels.  Every other edge in the tree has a parent so there are $2m-1$ ways to label it without creating a local move of type 2.  Thus there are $2m\cdot (2m-1)^{\lambda_i-1}\cdot C_{\lambda_i-1}$ possibilities for labeling the $i^{th}$ subtree. 

Taking the product over all $k$ such subtrees, there are $(2m)^k\cdot (2m-1)^{n-k}\cdot \prod \limits_{i=1}^k C_{\lambda_i-1}$  labeled plane trees with no valid local moves of type $2$ corresponding to the composition $\lambda$.  Summing over all compositions $\lambda$ of $n$ gives the desired result.
\end{proof}

The other main result of this section is that the ratio of valid primary sequences to all primary sequences approaches zero as $n$ increases. 

\begin{theorem}\label{theorem: limiting zero}
Fix $m$ and let $S(n,m)$ denote all possible words of length $2n$ over a complementary alphabet of size $2m$. Then
\[\lim_{n \to \infty} \frac{|\mathcal{P}(n,m)|}{|S(n,m)|} = 0.\]
\end{theorem}
\revised{
\begin{proof}
When $m=1$ the number of valid plane sequences $|\mathcal{P}(n,1)|$ is at most $\binom{2n}{n}$ since each valid plane sequence must have the same number of letters $B$ as $\overline{B}$.  (In fact $|\mathcal{P}(n,1)| = \binom{2n}{n}$.)  There are a total of $2^{2n}$ words in $\{B,\overline{B}\}$ of length $2n$. By Stirling's approximation for $n!$ we know
$$\lim \limits_{n\to \infty}\frac{\binom{2n}{n}}{2^{2n}} =\lim \limits_{n\to \infty}  \frac{\frac{4^n}{\sqrt{\pi n}}}{2^{2n}}= \lim \limits_{n \to \infty}\frac{1}{\sqrt{\pi n}} = 0.$$

%

\rerevised{
When $m>1$ we obtain an upper bound on $|\mathcal{P}(n,m)|$ by overcounting the set of valid plane trees $T_0$ that can be produced by the greedy algorithm. The greedy algorithm produces different outputs on different valid sequences because there is only one way to read the sequence off of the labeled plane tree.  We use the total number of valid plane trees as our upper bound; this is a strict overcount because some valid plane trees have a valid local move of type 2, unlike output of the greedy algorithm.  There are $(2m)^n C_n$ valid plane trees since there are $C_n$ plane trees with $n$ edges, $2m$ ways to pick the letter labeling each left-half-edge, and $1$ way to pick the complementary letters on the right-half-edges.  There are a total of $(2m)^{2n}$ words in the alphabet of length $2n$.}  Thus
\[\lim_{n \to \infty} \frac{|\mathcal{P}(n,m)|}{|S(n,m)|} <  \lim_{n \to \infty} \frac{(2m)^n C_n}{(2m)^{2n}} = \lim_{n \to \infty} \frac{1}{n+1} \frac{\binom{2n}{n}}{(2m)^n} \rerevised{\leq} \lim_{n \to \infty} \frac{1}{n+1} \frac{\binom{2n}{n}}{2^{2n}} \]
using the definition of Catalan numbers and the fact that $m \geq 2$.   The limit of the ratio $\frac{\binom{2n}{n}}{2^{2n}}$ was just computed to be zero, so as $n$ grows we obtain
\[\lim_{n \to \infty} \frac{|\mathcal{P}(n,m)|}{|S(n,m)|} <  \lim_{n \to \infty} \frac{1}{n+1} \frac{\binom{2n}{n}}{2^{2n}} < \lim_{n \to \infty} \frac{\binom{2n}{n}}{2^{2n}}=0\]
which proves the claim.
\end{proof}}

\section{Acknowledgements}
\label{sec:ack}
The authors gratefully acknowledge useful conversations with Christine Heitsch, very helpful comments from the referee, and the support of the Smith College Center for Women in Mathematics.

\nocite{*}
\bibliography{mybib}{}

\end{document}